\documentclass[10pt,a4paper,oneside]{article}
\usepackage[english]{babel}
\usepackage{a4wide,amsmath,amsthm,amssymb,url,graphicx,xspace,algorithm,algorithmic,pgf}
\usepackage[latin1]{inputenc}
\usepackage{enumitem}
\usepackage{multirow}
\usepackage{hyperref}
\usepackage{tabularx}

\def\F{\mathbb{F}}

\def\N{\mathbb{N}}

\DeclareMathOperator{\PG}{PG}

\theoremstyle{definition}
\newtheorem{theorem}{Theorem}[section]
\newtheorem{lemma}[theorem]{Lemma}
\newtheorem{definition}[theorem]{Definition}
\newtheorem{remark}[theorem]{Remark}
\newtheorem{corollary}[theorem]{Corollary}
\newtheorem{notation}[theorem]{Notation}
\newtheorem{example}[theorem]{Example}

\newcommand{\comments}[1]{}
\newcommand{\gs}[3]{\genfrac{[}{]}{0pt}{0}{#1}{#2}_{#3}}

\author{Maarten De Boeck}
\title{The second largest Erd\H{o}s-Ko-Rado sets of generators of the hyperbolic quadrics $\mathcal{Q}^{+}(4n+1,q)$}
\date{}
\begin{document}
\maketitle

\begin{abstract}
  An Erd\H{o}s-Ko-Rado set of generators of a hyperbolic quadric is a set of generators which are pairwise not disjoint. In this article we classify the second largest maximal Erd\H{o}s-Ko-Rado set of generators of the hyperbolic quadrics $\mathcal{Q}^{+}(4n+1,q)$, $q\geq3$.
\end{abstract}

\section{Introduction}
\label{sec:introduction}

\subsection{Hyperbolic quadrics}

Finite classical polar spaces are incidence structures consisting of the subspaces of a vector space over a finite field that are totally isotropic with respect to a certain non-degenerate sesquilinear or quadratic form. Thereby incidence is defined by the inclusion relation. Their maximal totally isotropic subspaces are called generators. Since the geometry of the subspaces of a vector space is a projective space, these classical polar spaces can be considered as substructures of a projective space. The projective space of dimension $n$ over the finite field $\F_{q}$ will be denoted by $\PG(n,q)$. Throughout this paper we will use the projective dimension for subspaces of polar spaces: a point is $0$-dimensional, a line is $1$-dimensional, ...
\par If such a classical polar space arises from a non-degenerate quadratic form it is called a quadric. As a substructure of the projective space it is defined by a  non-degenerate homogeneous quadratic equation of degree $2$. According to this quadratic form (or quadratic equation), three types of quadrics are distinguished: elliptic quadrics, parabolic quadrics and hyperbolic quadrics.
The quadrics in $\PG(2m+1,q)$ which can be defined by the standard equation
\[
  X_{0}X_{1}+X_{2}X_{3}+\cdots+X_{2m}X_{2m+1}=0\;,
\]
after projective transformation, are the hyperbolic quadrics. A hyperbolic quadric in $\PG(2m+1,q)$ is denoted by $\mathcal{Q}^{+}(2m+1,q)$.
\par An extensive introduction to quadrics can be found in \cite[Chapter 22]{ggg}. We list some basic facts which will be used in this article without reference.

\begin{theorem}
  Let $\mathcal{Q}^{+}(2m+1,q)$ be a hyperbolic quadric in $\PG(2m+1,q)$ and let $\Omega$ be its set of generators.
  \begin{itemize}
    \item The generators of $\mathcal{Q}^{+}(2m+1,q)$ are $m$-spaces.
    \item The number of elements of $\Omega$ is $\prod^{m}_{i=0}(q^{i}+1)$.
    \item The equivalence relation $\sim$ on $\Omega$, defined by $\pi\sim\pi'\Leftrightarrow\dim(\pi\cap\pi')=m\pmod{2}$ for any $\pi,\pi'\in\Omega$, partitions $\Omega$ in two generator classes of the same size, commonly called the Latin and the Greek generators.
    \item Let $\pi_{i}$ be an $i$-dimensional subspace of $\mathcal{Q}^{+}(2m+1,q)$, $0\leq i\leq m$. The tangent space $T_{\pi_{i}}(\mathcal{Q}^{+}(2m+1,q))$ to $\mathcal{Q}^{+}(2m+1,q)$ in $\pi_{i}$ is a $(2m-i)$-space in $\PG(2m+1,q)$. The intersection $T_{\pi_{i}}(\mathcal{Q}^{+}(2m+1,q))\cap\mathcal{Q}^{+}(2m+1,q)$ is a cone with vertex $\pi_{i}$ and base a hyperbolic quadric $\mathcal{Q}^{+}(2(m-i)-1,q)$.
  \end{itemize}
\end{theorem}

To simplify some of the notations in this article, we recall the definition of the Gaussian binomial coefficient:
\[
  \gs{n}{k}{q}=\prod_{i=1}^{k}\left(\frac{q^{n+1-i}-1}{q^i-1}\right)\;.
\]
The number of $k$-spaces in $\PG(n,q)$ is given by $\gs{n+1}{k+1}{q}$.

\subsection{\texorpdfstring{Erd\H{o}s-Ko-Rado}{Erdos-Ko-Rado} sets}

The original Erd\H{o}s-Ko-Rado theorem was stated in \cite{ekr}.
\begin{theorem}[{\cite[Theorem 1]{ekr}}]
  If $\mathcal{S}$ is a family of subsets of size $k$ in a set $\Omega$ with $|\Omega|=n$ and $n\geq 2k$, such that the elements of $\mathcal{S}$ are pairwise not disjoint, then $|\mathcal{S}|\leq \binom{n-1}{k-1}$.
\end{theorem}
In honour of this result, a family of subsets of fixed size in a set, pairwise not disjoint, is called an Erd\H{o}s-Ko-Rado set.  Note that a point-pencil, a family of subsets of size $k$ through a fixed element, meets the upper bound of the previous theorem. It was proved that this is the unique example meeting this upper bound if $n\geq 2k+1$. This is part of the following result which also classifies the second largest Erd\H{o}s-Ko-Rado set.
\begin{theorem}[\cite{hm}]\label{hiltonmilner}
  Let $\Omega$ be a set of size $n$ and let $\mathcal{S}$ be an Erd\H{o}s-Ko-Rado set of $k$-subsets in $\Omega$, $k\geq3$ and $n\geq2k+1$. If there is no element in $\Omega$ which is contained in all subsets in $\mathcal{S}$, then
  \[
    |\mathcal{S}|\leq\binom{n-1}{k-1}-\binom{n-k-1}{k-1}+1.
  \]
  Moreover, equality holds if and only if 
  \begin{itemize}
    \item either $\mathcal{S}$ is the union of $\{F\}$, for some fixed $k$-subset $F$, and the set of all $k$-subsets $G$ of $\Omega$ containing a fixed element $x\notin F$, such that $G\cap F\neq\emptyset$,
    \item or else $k=3$ and $\mathcal{S}$ is the set of all subsets of size $3$ having an intersection of size at least $2$ with a fixed subset $F$ of size $3$.
  \end{itemize}
\end{theorem}

This problem has been generalised to different geometrical structures, such as vector spaces, polar spaces and designs. We refer to the survey articles \cite{bbsw,ds} and for detailed reading to \cite{fw,psv,ran,tan} among others.
\par In this article we will study the Erd\H{o}s-Ko-Rado problem for generators of hyperbolic quadrics.
\begin{definition}
  Let $\mathcal{Q}^{+}(2m+1,q)$ be a hyperbolic quadric and let $\Omega$ be its set of generators. An Erd\H{o}s-Ko-Rado set $\mathcal{S}$ of generators of $\mathcal{Q}^{+}(2m+1,q)$ is a subset of $\Omega$ such that any two elements of $\mathcal{S}$ have a nonempty intersection. If $\mathcal{S}$ is not extendable regarding this condition it is called maximal.
\end{definition}

The Erd\H{o}s-Ko-Rado problem now asks for the classification of the large maximal Erd\H{o}s-Ko-Rado sets. 
\par For a hyperbolic quadric $\mathcal{Q}^{+}(2m+1,q)$, its Erd\H{o}s-Ko-Rado sets heavily depend on the parity of $m$. If $m$ is even, any two generators belonging to the same class meet in at least a point; if $m$ is odd, any two generators belonging to different classes meet in at least a point. In the latter case, it is therefore sufficient to study Erd\H{o}s-Ko-Rado sets of generators of one class. In this article we study the former case. We give an example of a maximal Erd\H{o}s-Ko-Rado set.

\begin{example}\label{oneclass}
  Let $\mathcal{Q}^{+}(4n+1,q)$ be a hyperbolic quadric and let $\mathcal{S}$ be the set of generators of one class. Since any two generators belonging to the same class cannot be disjoint, $\mathcal{S}$ is an Erd\H{o}s-Ko-Rado set. Moreover, it is maximal since no generator of the other class can meet all generators of the class corresponding to $\mathcal{S}$.
  \par Note that $|\mathcal{S}|=\prod^{2n}_{i=1}(q^{i}+1)$, half of the total number of generators.
\end{example}

This is an important example, as the next theorem shows.

\begin{theorem}[{\cite[Theorem 9 and Theorem 16]{psv}}]\label{largest}
  Let $\mathcal{S}$ be an Erd\H{o}s-Ko-Rado set of generators of $\mathcal{Q}^{+}(4n+1,q)$. Then, $|\mathcal{S}|\leq\prod^{2n}_{i=1}(q^{i}+1)$. Furthermore, if  $|\mathcal{S}|=\prod^{2n}_{i=1}(q^{i}+1)$, then $\mathcal{S}$ is the Erd\H{o}s-Ko-Rado set described in Example \ref{oneclass}.
\end{theorem}

So, for general hyperbolic quadrics $\mathcal{Q}^{+}(4n+1,q)$ the largest Erd\H{o}s-Ko-Rado sets of generators have been classified. In fact, in \cite{psv}, the largest Erd\H{o}s-Ko-Rado sets of generators were classified for all finite classical polar spaces except the Hermitian varieties $\mathcal{H}(4n+1,q^{2})$, $n\geq2$. For those, the best result was obtained in \cite{im}.
\par We return to hyperbolic quadrics. For the smallest nontrivial case, a complete classification of Erd\H{o}s-Ko-Rado sets of generators is known.

\begin{theorem}[{\cite[Theorem 3.5]{bh}}]\label{q+5q}
  Let $\mathcal{S}$ be a maximal Erd\H{o}s-Ko-Rado set of generators of $\mathcal{Q}^{+}(5,q)$. Then, $\mathcal{S}$ is the set of all generators of one class, the set of generators meeting a fixed generator in at least a line, or the set of all generators through a fixed point.
\end{theorem}

These examples contain $q^{3}+q^{2}+q+1$, $q^{2}+q+2$ and $2q+2$ generators, respectively. We note that the proof of this theorem in \cite{bh} relies on the dual polar graph $D_{3}(q)$. This is the graph whose vertices are the generators of $\mathcal{Q}^{+}(5,q)$; two vertices are adjacent if and only the corresponding generators meet in a line (in general, in a hyperplane of a generator). The authors of \cite{bh} study the maximal $\{0,1,2\}$-cliques in this graph (and other dual polar graphs); these are vertex sets such that any two vertices in the set are at distance $0$, $1$ or $2$ in the graph. So, a $\{0,1,2\}$-clique in $D_{3}(q)$ corresponds to a set of generators (planes) of $\mathcal{Q}^{+}(5,q)$ that pairwise meet in a plane, line or point, i.e. an Erd\H{o}s-Ko-Rado set.
\par Since the largest Erd\H{o}s-Ko-Rado sets of generators have been classified, we focus on other maximal Erd\H{o}s-Ko-Rado sets of generators. In this article we classify the second largest maximal Erd\H{o}s-Ko-Rado sets of generators of $\mathcal{Q}^{+}(4n+1,q)$. This result can be found in Theorem \ref{maintheorem}. This is an analogue of Theorem \ref{hiltonmilner}.

\section{Counting results}
\label{sec:countingresults}

We recall two counting results, one about subspaces and one about generators.

\begin{theorem}[{\cite[Section 170]{seg}}]\label{skewsubspaces}
  The number of $j$-spaces skew to a fixed $k$-space in $\PG(n,q)$ equals $q^{(k+1)(j+1)}\gs{n-k}{j+1}{q}$.
\end{theorem}

\begin{theorem}[{\cite[Corollary 5]{kms}}]\label{skewgenerators}
  Let $\mathcal{Q}^{+}(2m+1,q)$ be a hyperbolic quadric, and let $\pi_{1}$ and $\pi_{2}$ be two generators of $\mathcal{Q}^{+}(2m+1,q)$ meeting in a $j$-dimensional space. The number of generators skew to both $\pi_{1}$ and $\pi_{2}$ equals
  \[
    b^{m}_{j}:=\begin{cases}
                q^{2\binom{(m+j)/2+1}{2}-\binom{j+1}{2}}\prod^{(m-j)/2}_{i=1}(q^{2i-1}-1)& m\equiv j\pmod{2}\\
                0 & m\equiv j+1\pmod{2}
              \end{cases}.
  \]
\end{theorem}

Now, we present a new counting result.
\begin{lemma}\label{meetingone}
  Let $\mathcal{Q}^{+}(4n+1,q)$ be a hyperbolic quadric and let $\pi_{1}$ and $\pi_{2}$ be two generators of the same class on $\mathcal{Q}^{+}(4n+1,q)$ meeting in a $j$-dimensional space, $0\leq j\leq 2n$ and $j$ even. The number of generators meeting $\pi_{1}$, but not $\pi_{2}$, equals
  \[
    v^{n}_{j}:=\sum^{n-1}_{i=\frac{j}{2}}q^{(2n-2i)(j+1)}\gs{2n-j}{2n-2i}{q}b^{2i}_{j}\;.
  \]
\end{lemma}
\begin{proof}
  All generators belonging to the same class as $\pi_{1}$ and $\pi_{2}$, meet both, hence cannot meet precisely one of them. Let $\pi$ be a generator of the other class that meets $\pi_{1}$. The intersection $\tau=\pi_{1}\cap\pi$ is a $(2n-2i-1)$-space, for an $i$ fulfilling $\frac{j}{2}\leq i\leq n-1$. Let $\overline{\tau}$ be the tangent space in $\tau$ to $\mathcal{Q}^{+}(4n+1,q)$; it is $(2n+2i+1)$-dimensional. The tangent space $\overline{\tau}$ contains $\pi_{1}$ and meets $\pi_{2}$ in a $(2i)$-space through $\pi_{1}\cap\pi_{2}$. The intersection $\overline{\tau}\cap\mathcal{Q}^{+}(4n+1,q)$ is a cone with vertex $\tau$ and base a hyperbolic quadric $\mathcal{Q}^{+}(4i+1,q)$. We can choose the ambient space $\sigma$ of this base $\mathcal{Q}^{+}(4i+1,q)$ to contain $\overline{\tau}\cap\pi_{2}$. 
  \par Any generator through $\tau$ now corresponds to a generator of this base quadric $\mathcal{Q}^{+}(4i+1,q)$. Moreover, a generator through $\tau$ meeting $\pi_{1}$ in $\tau$ and disjoint to $\pi_{2}$ corresponds to a generator of $\mathcal{Q}^{+}(4i+1,q)$ skew to both $\overline{\tau}\cap\pi_{2}=\sigma\cap\pi_{2}$ and $\sigma\cap\pi_{1}$, which both are generators of the base quadric. Since $(\sigma\cap\pi_{1})\cap(\sigma\cap\pi_{2})=\pi_{1}\cap\pi_{2}$, the number of such generators equals $b^{2i}_{j}$.
  \par So, the total number of generators meeting $\pi_{1}$, but not $\pi_{2}$, equals
  \[
    \sum^{n-1}_{i=\frac{j}{2}}q^{(2n-2i)(j+1)}\gs{2n-j}{2n-2i}{q}b^{2i}_{j}\;.
  \]
  Here we used the result from Theorem \ref{skewsubspaces} to count the number of $(2n-2i-1)$-spaces in $\pi_{1}$ that are skew to $\pi_{1}\cap\pi_{2}$.
\end{proof}

\begin{corollary}\label{wnt}
  Let $\mathcal{Q}^{+}(4n+1,q)$ be a hyperbolic quadric and let $\pi_{1}$ and $\pi_{2}$ be two generators of the same class on $\mathcal{Q}^{+}(4n+1,q)$ meeting in a $2(n-t)$-dimensional space, $0\leq t\leq n$. The number of generators not meeting both $\pi_{1}$ and $\pi_{2}$ equals
  \[
    W^{n}_{t}(q):=q^{2n^{2}+n-t^{2}}\left(\prod^{t}_{k=1}(q^{2k-1}-1)+2\sum^{t}_{i=1}\gs{2t}{2i}{q}q^{i^{2}+i-2it}\prod^{t-i}_{k=1}(q^{2k-1}-1)\right)\;.
  \]
\end{corollary}
\begin{proof}
   Using the notations from Theorem \ref{skewgenerators} and Lemma \ref{meetingone}, we find that the number of generators not meeting both $\pi_{1}$ and $\pi_{2}$ equals $b^{2n}_{2(n-t)}+2v^{n}_{2(n-t)}$. Using the results from Theorem \ref{skewgenerators} and Lemma \ref{meetingone}, we find that
   \begin{align*}
     W^{n}_{t}(q)&=b^{2n}_{2(n-t)}+2v^{n}_{2(n-t)}\\
     &=q^{2n^{2}+n-t^{2}}\prod^{t}_{k=1}(q^{2k-1}-1)+2\sum^{n-1}_{i=n-t}q^{(2n-2i)(2n-2t+1)}\gs{2t}{2n-2i}{q}b^{2i}_{2(n-t)}\\
     &=q^{2n^{2}+n-t^{2}}\prod^{t}_{k=1}(q^{2k-1}-1)+2\sum^{t}_{i=1}q^{2i(2n-2t+1)}\gs{2t}{2i}{q}b^{2(n-i)}_{2(n-t)}\\
     &=q^{2n^{2}+n-t^{2}}\prod^{t}_{k=1}(q^{2k-1}-1)+2\sum^{t}_{i=1}q^{2i(2n-2t+1)}\gs{2t}{2i}{q}q^{2(n-i)^{2}+(n-i)-(t-i)^{2}}\prod^{t-i}_{k=1}(q^{2k-1}-1)\\
     &=q^{2n^{2}+n-t^{2}}\left(\prod^{t}_{k=1}(q^{2k-1}-1)+2\sum^{t}_{i=1}\gs{2t}{2i}{q}q^{i^{2}+i-2it}\prod^{t-i}_{k=1}(q^{2k-1}-1)\right)\;.\qedhere
   \end{align*}
\end{proof}

\section{Classification results}
\label{sec:classificationresults}

\begin{example}\label{secondexample}
  Let $\pi$ be a generator of the hyperbolic quadric $\mathcal{Q}^{+}(4n+1,q)$ and let $\mathcal{G}$ be the generator class not containing $\pi$. Now, let $\mathcal{S}$ be the set containing $\pi$ and all generators of $\mathcal{G}$ that meet $\pi$. All elements of $\mathcal{S}\setminus\{\pi\}$ meet $\pi$ and since all generators of the same class have a nontrivial intersection on this hyperbolic quadric, they also meet each other. Hence, $\mathcal{S}$ is an Erd\H{o}s-Ko-Rado set. None of the generators in $\mathcal{G}\setminus \mathcal{S}$ extends $\mathcal{S}$ to a larger Erd\H{o}s-Ko-Rado set. Indeed a generator in $\mathcal{G}\setminus \mathcal{S}$ is disjoint from $\pi$. Also, for every generator $\pi'$ in the same class of $\pi$ we can find a generator in $\mathcal{S}$ not meeting $\pi'$. Consequently, this Erd\H{o}s-Ko-Rado set is maximal.
  \par The number of generators in $\mathcal{S}$ equals $(|\mathcal{G}|-b^{2n}_{2n})+1=\prod^{2n}_{i=1}(q^{i}+1)-q^{2n^{2}+n}+1$.
\end{example}

\begin{lemma}\label{boundonthird}
  Let $\mathcal{S}$ be a maximal Erd\H{o}s-Ko-Rado set of generators of a hyperbolic quadric $\mathcal{Q}^{+}(4n+1,q)$. If $\mathcal{S}$ is not an Erd\H{o}s-Ko-Rado set as described in Example \ref{oneclass} or Example \ref{secondexample}, then it contains at most $2\prod^{2n}_{k=1}(q^{k}+1)-2\min\{W^{n}_{t}(q)\mid1\leq t\leq n\}$ generators.
\end{lemma}
\begin{proof}
  Since $\mathcal{S}$ differs from the Erd\H{o}s-Ko-Rado sets described in Example \ref{oneclass} and Example \ref{secondexample}, it contains at least two distinct generators of both classes. Let $\pi_{1},\pi_{2}\in \mathcal{S}$ be two generators of the same class, whose intersection is $2(n-t)$-dimensional, $t\geq1$. Then $\mathcal{S}$ contains at most
  \[
    \prod^{2n}_{k=1}(q^{k}+1)-W^{n}_{t}(q)
  \]
  generators of the other class. The statement immediately follows.
\end{proof}

\begin{notation}
  The function $f_{t}(q)$, $t\geq1$, is defined in the following way:
  \[
    f_{t}(q):=q^{t^{2}-t}\prod^{t}_{k=1}(q^{2k-1}-1)+2\sum^{t-1}_{i=0}\gs{2t}{2i}{q}q^{i^{2}-i}\prod^{i}_{k=1}(q^{2k-1}-1)\;.
  \]
  Using Corollary \ref{wnt}, we see that $W^{n}_{t}(q)=q^{(n+t)(2n-2t+1)}f_{t}(q)$.
\end{notation}

It should be noted that $f_{t}(q)$ is independent of $n$, however closely related to $W^{n}_{t}(q)$. We calculate $f_{t}(q)$ for some small values of $t$.
\begin{align*}
  f_{1}(q)&=q+1\\
  f_{2}(q)&=q^{6}+q^{5}+q^{3}-q^{2}\\
  f_{3}(q)&=q^{15}+q^{14}+q^{12}-q^{11}+q^{10}-q^{9}-q^{7}+q^{6}\\
\end{align*}

We prove an inequality between these functions.

\begin{lemma}\label{ftineq}
  For every $t\geq1$ and $q\geq2$, the inequality $f_{t+1}(q)>q^{4t+1}f_{t}(q)$ holds.
\end{lemma}
\begin{proof}
  We carry out the following calculations.
  \begin{align*}
    f_{t+1}(q)&=q^{(t+1)^{2}-(t+1)}\prod^{t+1}_{k=1}(q^{2k-1}-1)+2\sum^{t}_{i=0}\gs{2t+2}{2i}{q}q^{i^{2}-i}\prod^{i}_{k=1}(q^{2k-1}-1)\\
    &=q^{2t}(q^{2t+1}-1)\left(q^{t^{2}-t}\prod^{t}_{k=1}(q^{2k-1}-1)\right)+2\\&\qquad+2\sum^{t}_{i=1}\frac{(q^{2t+2}-1)(q^{2t+1}-1)}{q^{2i}-1}\gs{2t}{2i-2}{q}q^{i^{2}-i}\prod^{i-1}_{k=1}(q^{2k-1}-1)\\
    &=q^{2t}(q^{2t+1}-1)\left(q^{t^{2}-t}\prod^{t}_{k=1}(q^{2k-1}-1)\right)+2\\&\qquad+2\sum^{t-1}_{i=0}\frac{(q^{2t+2}-1)(q^{2t+1}-1)q^{2i}}{q^{2i+2}-1}\gs{2t}{2i}{q}q^{i^{2}-i}\prod^{i}_{k=1}(q^{2k-1}-1)\\
  \end{align*}
  Note that
  \[
    \frac{(q^{2t+2}-1)(q^{2t+1}-1)q^{2i}}{q^{2i+2}-1}=q^{4t+1}+\frac{q^{4t+1}-q^{2t+2i+2}-q^{2t+2i+1}+q^{2i}}{q^{2i+2}-1}>q^{4t+1}
  \]
  since $i\leq t-1$.
  Substituting both the equality (for $i=t-1$) and the inequality in the previous calculation, we find
  \begin{align*}
     f_{t+1}(q)&>q^{2t}(q^{2t+1}-1)\left(q^{t^{2}-t}\prod^{t}_{k=1}(q^{2k-1}-1)\right)+2q^{4t+1}\sum^{t-1}_{i=0}\gs{2t}{2i}{q}q^{i^{2}-i}\prod^{i}_{k=1}(q^{2k-1}-1)\\&\qquad+2\frac{q^{4t+1}-q^{4t}-q^{4t-1}+q^{2t-2}}{q^{2t}-1}\gs{2t}{2t-2}{q}q^{(t-1)(t-2)}\prod^{t-1}_{k=1}(q^{2k-1}-1)\\
     &=q^{4t+1}f_{t}(q)-q^{t^{2}+t}\prod^{t}_{k=1}(q^{2k-1}-1)\\&\qquad+2\frac{q^{4t+1}-q^{4t}-q^{4t-1}+q^{2t-2}}{(q^{2}-1)(q-1)}q^{(t-1)(t-2)}\prod^{t}_{k=1}(q^{2k-1}-1)\\
     &>q^{4t+1}f_{t}(q)+\left(2(q^{4t-2}-q^{4t-4}-q^{4t-5})-q^{4t-2}\right)q^{(t-1)(t-2)}\prod^{t}_{k=1}(q^{2k-1}-1)\\
     &\geq q^{4t+1}f_{t}(q)\;.
  \end{align*}
  Here we used that $q^{4t-2}-2q^{4t-4}-2q^{4t-5}\geq0$ for all $q\geq2$.
\end{proof}

The following inequality was proven in \cite{d}.

\begin{lemma}[{\cite[Corollary 3.3]{d}}]\label{basicinequality}
  Let $s,t,q\in\N$ be such that $s\leq t$ and $q\geq3$. If $(s,q)\neq(0,3)$, then
  \[
    \prod^{t}_{i=s}(q^{i}+1)\leq (q^{s}+2)q^{\binom{t+1}{2}-\binom{s+1}{2}}\;.
  \]
\end{lemma}

\begin{corollary}\label{theinequality}
  If $q\geq3$ and $n\geq1$, then $q^{2n^{2}+n}+2q^{2n^{2}+n-1}+1>\prod^{2n}_{k=1}(q^{k}+1)$.
\end{corollary}
\begin{proof}
  We apply Lemma \ref{basicinequality} for $(s,t)=(1,2n)$ and we find
  \[
    \prod^{2n}_{i=1}(q^{i}+1)\leq (q+2)q^{\binom{2n+1}{2}-1}-q^{\binom{2n}{2}}<(q+2)q^{2n^{2}+n-1}+1\;.\qedhere
  \]
\end{proof}

Now we prove the main theorem of this paper.

\begin{theorem}\label{maintheorem}
  The two largest types of maximal Erd\H{o}s-Ko-Rado sets of generators of a hyperbolic quadric $\mathcal{Q}^{+}(4n+1,q)$, $n\geq1$, $q\geq3$, are the Erd\H{o}s-Ko-Rado sets described in Example \ref{oneclass} and Example \ref{secondexample}.
\end{theorem}
\begin{proof}
  Let $\mathcal{S}$ be a maximal Erd\H{o}s-Ko-Rado set of generators different from the Erd\H{o}s-Ko-Rado sets described in Example \ref{oneclass} and Example \ref{secondexample}.  By Lemma \ref{boundonthird} we know that
  \[
    |\mathcal{S}|\leq 2\prod^{2n}_{k=1}(q^{k}+1)-2\min\{W^{n}_{t}(q)\mid1\leq t\leq n\}\;.
  \]
  Using Lemma \ref{ftineq}, we find that
  \[
    W^{n}_{t+1}(q)=q^{(n+t+1)(2n-2t-1)}f_{t+1}(q)>q^{(n+t)(2n-2t+1)}f_{t}(q)= W^{n}_{t}(q)\;,
  \]
  for all $1\leq t\leq n$. Hence,
  \[
    \min\{W^{n}_{t}(q)\mid1\leq t\leq n\}=W^{n}_{1}(q)=(q+1)q^{2n^{2}+n-1}\;.
  \]
  \par It is clear that the Erd\H{o}s-Ko-Rado set described in Example \ref{oneclass} is larger than the one described in Example \ref{secondexample}, so we compare the upper bound on $|\mathcal{S}|$ with the size of the Erd\H{o}s-Ko-Rado set described in Example \ref{secondexample}. The corresponding inequality
  \[
    \prod^{2n}_{i=1}(q^{i}+1)-q^{2n^{2}+n}+1> 2\prod^{2n}_{k=1}(q^{k}+1)-2(q+1)q^{2n^{2}+n-1} 
  \]
  is equivalent to
  \[
    q^{2n^{2}+n}+2q^{2n^{2}+n-1}+1>\prod^{2n}_{k=1}(q^{k}+1)\;.
  \]
  By Lemma \ref{theinequality} we know that this inequality is valid if $q\geq3$.
\end{proof}

\begin{remark}
  In the previous theorem the case $q=2$ is omitted since the key inequality $q^{2n^{2}+n}+2q^{2n^{2}+n-1}+1>\prod^{2n}_{k=1}(q^{k}+1)$ is not true for $q=2$ if $n\geq 2$. E.g. $2^{10}+2\cdot2^{9}+1=2049<2295=\prod^{4}_{k=1}(2^{k}+1)$.
  \par For $n=1$, Theorem \ref{q+5q} implies the previous result (also in the case $q=2$).
\end{remark}

\comments{
\begin{lemma}\label{reducedproblem}
  The two largest types of maximal Erd\H{o}s-Ko-Rado sets of generators of a hyperbolic quadric $\mathcal{Q}^{+}(4n+1,q)$, $n\geq1$, are the Erd\H{o}s-Ko-Rado sets described in Example \ref{oneclass} and Example \ref{secondexample} if
  \[
    q^{2n^{2}+n}+2q^{2n^{2}+n-1}+1>\prod^{2n}_{k=1}(q^{k}+1)\;.
  \]
\end{lemma}
}

\section{Some other examples of large \texorpdfstring{Erd\H{o}s-Ko-Rado}{Erdos-Ko-Rado} sets}

Next to the two examples of maximal Erd\H{o}s-Ko-Rado sets of generators, presented in Example \ref{oneclass} and Example \ref{secondexample}, we also know the \emph{point-pencil}. This is the set of all generators through a fixed point. For many geometries, the point-pencil is the largest Erd\H{o}s-Ko-Rado set, e.g. for vector spaces (\cite{fw,tan}), for designs (\cite{ran}), for the polar spaces $\mathcal{Q}^{-}(2n+1,q)$, $\mathcal{W}(4n+3,q)$, ... (\cite{psv}). By Theorem \ref{largest} we know that this is not true for hyperbolic quadrics $\mathcal{Q}^{+}(4n+1,q)$. In this case the point-pencil contains
\[
  \prod^{2n-1}_{i=0}(q^{i}+1)=2\prod^{2n-1}_{i=1}(q^{i}+1)\in\Theta(q^{2n^{2}-n})
\]
generators. Recall that the Erd\H{o}s-Ko-Rado sets described in Example \ref{oneclass} and Example \ref{secondexample} contain $\prod^{2n}_{i=1}(q^{i}+1)\in\Theta(q^{2n^{2}+n})$ generators and $\prod^{2n}_{i=1}(q^{i}+1)-q^{2n^{2}+n}+1\in\Theta(q^{2n^{2}+n-1})$ generators, respectively. So, the point-pencil is much smaller in this case. Here we used the $\Theta$-notation: $f\in\Theta(g)$ iff $0<\lim_{x\to\infty}\frac{f(x)}{g(x)}<\infty$. In other terms, $f=kg+h$, with $k>0$ a real number and $\lim_{x\to\infty}\frac{h(x)}{g(x)}=0$.
\par In this section we will present some more Erd\H{o}s-Ko-Rado sets of generators of $\mathcal{Q}^{+}(4n+1,q)$ whose size is larger than the size of a point-pencil. First we give a counting result.

\begin{lemma}\label{generatorsskewtosubspace}
  Let $m\geq0$ and $k\geq -1$ be two integers such that $k<m$. Let $\Omega$ be one of the two generator classes of a hyperbolic quadric $\mathcal{Q}^{+}(2m+1,q)$. The number of generators in $\Omega$ skew to a fixed $k$-space on the quadric equals
  \[
    \frac{1}{2}\left(\prod^{m-k-1}_{i=0}(q^{i}+1)\right)q^{\frac{1}{2}(k+1)(2m-k)}=:w_{m,k}\;.
  \]
  The empty space is considered to have dimension $-1$.
\end{lemma}
\begin{proof}
  Let $\pi$ be a $k$-dimensional subspace of $\mathcal{Q}^{+}(2m+1,q)$. We prove this lemma by using induction on $k$. If $k=-1$, then $\pi$ is the empty space. The number of generators of $\Omega$ skew to the empty space is the total number of generators of $\Omega$, which equals $w_{m,-1}$.
  \par Now, we assume that the lemma is proved for all subspaces of dimension at most $k-1$; we will prove it for a $k$-dimensional space $\pi$. The subspace $\pi$ contains $\gs{k+1}{i+1}{q}$ subspaces of dimension $i$, $0\leq i\leq k$. Let $\sigma$ be such an $i$-space and let $T_{\sigma}(\mathcal{Q}^{+}(2m+1,q))$ be its tangent space to $\mathcal{Q}^{+}(2m+1,q)$. We know that $\mathcal{Q}^{+}(2m+1,q)\cap T_{\sigma}(\mathcal{Q}^{+}(2m+1,q))$ is a cone with vertex $\sigma$ and base a hyperbolic quadric $\mathcal{Q}^{+}(2m-2i-1,q)$. The $k$-space $\pi$ corresponds to a $(k-i-1)$-space in this base. Arguing as in the proof of Lemma \ref{meetingone}, the number of generators of $\Omega$ meeting $\pi$ in precisely $\sigma$ equals $w_{m-i-1,k-i-1}$. Here, we note that the generators of $\Omega$ through $\sigma$ correspond to one of the two classes of generators of the base $\mathcal{Q}^{+}(2m-2i-1,q)$.
  \par So, the total number of generators of $\Omega$ skew to $\pi$ is independent of the choice for $\pi$, and equals
  \begin{align*}
    w_{m,k}&=\frac{1}{2}\prod^{m}_{j=0}(q^{j}+1)-\sum^{k}_{i=0}\gs{k+1}{i+1}{q}w_{m-i-1,k-i-1}\\
    &=\frac{1}{2}\prod^{m}_{j=0}(q^{j}+1)-\frac{1}{2}\sum^{k}_{i=0}\gs{k+1}{i+1}{q}\left(\prod^{m-k-1}_{j=0}(q^{j}+1)\right)q^{\frac{1}{2}(k-i)(2m-k-i-1)}\\
    &=\frac{1}{2}\left(\prod^{m-k-1}_{j=0}(q^{j}+1)\right)\left(\prod^{m}_{j=m-k}(q^{j}+1)-\sum^{k}_{i=0}\gs{k+1}{i+1}{q}q^{\frac{1}{2}(k-i)(2m-k-i-1)}\right)\\
    &=\frac{1}{2}\left(\prod^{m-k-1}_{j=0}(q^{j}+1)\right)\left(\prod^{m}_{j=m-k}(q^{j}+1)-\sum^{k+1}_{i=1}\gs{k+1}{i}{q}q^{\frac{1}{2}(k-i+1)(2m-k-i)}\right)\\
    &=\frac{1}{2}\left(\prod^{m-k-1}_{j=0}(q^{j}+1)\right)\left(\prod^{m}_{j=m-k}(q^{j}+1)-q^{\frac{1}{2}(2m-k)(k+1)}\sum^{k+1}_{i=1}\gs{k+1}{i}{q}q^{\binom{i}{2}}q^{-mi}\right)\\
    &=\frac{1}{2}\left(\prod^{m-k-1}_{j=0}(q^{j}+1)\right)\left(\prod^{m}_{j=m-k}(q^{j}+1)-q^{\frac{1}{2}(2m-k)(k+1)}\left(\prod^{k}_{j=0}(q^{j-m}+1)-1\right)\right)\\
    &=\frac{1}{2}\left(\prod^{m-k-1}_{j=0}(q^{j}+1)\right)q^{\frac{1}{2}(k+1)(2m-k)}\;.\\
  \end{align*}
  In the penultimate transition we used the $q$-binomial theorem
  \[
    \prod^{n-1}_{l=0}(1+q^{l}t)=\sum^{n}_{l=0}q^{\binom{l}{2}}\gs{n}{l}{q}t^{l}\;.
  \]
  This calculation finishes the induction step.
\end{proof}

\begin{remark}\label{skewtogenerator}
  In the previous theorem, the case $m=k$ was not covered; in this case we count the number of generators of a fixed class skew to a given generator $\pi$. We already know that this number will be dependent on the class of $\pi$ and the parity of $m$. Using the observation before Example \ref{oneclass} and Theorem \ref{skewgenerators}, we can state the following result. If $m$ is even, then no generators of the class of $\pi$ are skew to $\pi$ and $b^{m}_{m}=q^{\binom{m+1}{2}}$ generators of the other class are skew to $\pi$. If $m$ is odd, then $b^{m}_{m}=q^{\binom{m+1}{2}}$ generators of the class of $\pi$ are skew to $\pi$ and no generators of the other class are skew to $\pi$.
  \par It is an immediate consequence of the previous theorem that the total number of generators skew to a fixed $k$-space on a hyperbolic quadric $\mathcal{Q}^{+}(2m+1,q)$, $k<m$, equals
  \[
    \left(\prod^{m-k-1}_{i=0}(q^{i}+1)\right)q^{\frac{1}{2}(k+1)(2m-k)}=2w_{m,k}\;.
  \]
\end{remark}

We now introduce some new examples of large maximal Erd\H{o}s-Ko-Rado sets of the hyperbolic quadric $\mathcal{Q}^{+}(4n+1,q)$.

\begin{example}\label{ikj}
  Consider the hyperbolic quadric $\mathcal{Q}^{+}(4n+1,q)$ and let $\tau$ be a fixed $k$-space on it, $0\leq k\leq2n$. Denote the two classes of generators by $\Omega_{1}$ and $\Omega_{2}$. Let $\mathcal{S}$ be the union of the set of generators of $\Omega_{1}$ meeting $\tau$ in a subspace of dimension at least $j$, $0\leq j\leq k$, and the set of generators of $\Omega_{2}$ meeting $\tau$ in a subspace of dimension at least $k-j$. It is immediate that the elements of $\mathcal{S}$ pairwise intersect. Consequently, $\mathcal{S}$ is an Erd\H{o}s-Ko-Rado set. We denote this type of Erd\H{o}s-Ko-Rado sets by $I_{k,j}$. Note that not all these types are different. If $k<2n$, then $I_{k,j}$ and $I_{k,k-j}$ describe similar sets of generators. Also $I_{2n-1,2j-1}$, $I_{2n,2j-1}$ and $I_{2n,2j}$ describe similar sets of generators, $1\leq j\leq n$.
  \par We show that an Erd\H{o}s-Ko-Rado set $\mathcal{S}$ of type $I_{k,j}$ is maximal. Assume that we can find a generator $\pi$ in $\Omega_{1}\setminus\mathcal{S}$ which meets all generators of $\mathcal{S}$. Since $\pi\notin\mathcal{S}$, we know that $\dim(\pi\cap\tau)<j$. So, we can find a $(k-j)$-space $\tau'$ in $\tau$ disjoint to $\pi\cap\tau$. We know that all generators of $\Omega_{2}$ containing $\tau'$ belong to $\mathcal{S}$. Let $T_{\tau'}$ be the tangent space in $\tau'$ to $\mathcal{Q}^{+}(4n+1,q)$. It is $(4n-k+j)$-dimensional and meets $\pi$ in a $(2n-k+j-1)$-space disjoint to $\tau'$. The intersection $\mathcal{Q}^{+}(4n+1,q)\cap T_{\tau'}$ is a cone with vertex $\tau'$ and base a hyperbolic quadric $\mathcal{Q}_{1}:=\mathcal{Q}^{+}(2(2n-k+j-1)+1,q)$. We can choose this basis such that it contains $\pi'=T_{\tau'}\cap\pi$. Moreover $T_{\tau'}\cap\pi$ is a generator of $\mathcal{Q}_{1}$. The set of generators of $\Omega_{2}$ through $\tau'$, all in $\mathcal{S}$, correspond to the set of generators of one class of $\mathcal{Q}_{1}$. We denote this class by $\Omega'_{2}$. If $k-j$ is even, then $\pi'$ also belongs to $\Omega'_{2}$. However, $2n-k+j-1$ is odd and by an observation made in Remark \ref{skewtogenerator}, we know that we can find a generator $\sigma\in\Omega'_{2}$ skew to $\pi'$. Then $\left\langle\pi',\tau'\right\rangle$ is a generator in $\mathcal{S}$ skew to $\pi$. If $k-j$ is odd, then $\pi'$ belongs to $\Omega'_{1}$, the other class of generators of $\mathcal{Q}_{1}$. In this case, $2n-k+j-1$ is even and by an observation made in Remark \ref{skewtogenerator}, we know that we can find a generator $\sigma\in\Omega'_{2}$ skew to $\pi'$. Then $\left\langle\pi',\tau'\right\rangle$ is a generator in $\mathcal{S}$ skew to $\pi$. The arguments for $\pi\in\Omega_{2}\setminus\mathcal{S}$ are analogous.
  \par Now, we count the number of generators in an Erd\H{o}s-Ko-Rado set $\mathcal{S}$ of type $I_{k,j}$, $k<2n$. We use the result from Lemma \ref{generatorsskewtosubspace}.
  \begin{align*}
    |\mathcal{S}|&=\sum^{k}_{i=j}\gs{k+1}{i+1}{q}w_{2n-i-1,k-i-1}+\sum^{k}_{i=k-j}\gs{k+1}{i+1}{q}w_{2n-i-1,k-i-1}\\
    &=\sum^{k}_{i=j}\gs{k+1}{i+1}{q}\frac{1}{2}\left(\prod^{2n-k-1}_{p=0}(q^{p}+1)\right)q^{\frac{1}{2}(k-i)(4n-k-i-1)}\\&\qquad\qquad+\sum^{k}_{i=k-j}\gs{k+1}{i+1}{q}\frac{1}{2}\left(\prod^{2n-k-1}_{p=0}(q^{p}+1)\right)q^{\frac{1}{2}(k-i)(4n-k-i-1)}\\
    &=\frac{1}{2}\left(\prod^{2n-k-1}_{p=0}(q^{p}+1)\right)\left[\sum^{k}_{i=j}\gs{k+1}{i+1}{q}q^{\frac{1}{2}(k-i)(4n-k-i-1)}+\sum^{k}_{i=k-j}\gs{k+1}{i+1}{q}q^{\frac{1}{2}(k-i)(4n-k-i-1)}\right]\;.\\
  \end{align*}
  Using this result, we can see that an Erd\H{o}s-Ko-Rado set of type $I_{k,j+1}$ is larger than an Erd\H{o}s-Ko-Rado set of type $I_{k,j}$ if and only if $2j+1-k>0$, $k<2n$. Using this and the above mentioned equality between $I_{k,j}$ and $I_{k,k-j}$, we find that the largest among these Erd\H{o}s-Ko-Rado sets are the ones of type $I_{k,k}$, which are also the ones of type $I_{k,0}$. Those contain
  \[
    \prod^{2n}_{i=1}(q^{i}+1)-\frac{1}{2}(q^{\frac{1}{2}(k+1)(4n-k)}-1)\prod^{2n-k-1}_{i=0}(q^{i}+1)\in\Theta(q^{2n^{2}-n+k})
  \]
  generators. In this computation we used the $q$-binomial theorem. Since the Erd\H{o}s-Ko-Rado sets of type $I_{2n,2j-1}$ and $I_{2n,2j}$ are equal to the Erd\H{o}s-Ko-Rado sets $I_{2n-1,2j-1}$, $1\leq j\leq n$, we can use the above formulas to compute their number of elements as well.
  \par So, the only type $I_{k,j}$ of Erd\H{o}s-Ko-Rado sets whose size has not been computed above is $I_{2n,0}$. However, the Erd\H{o}s-Ko-Rado sets of type $I_{2n,0}$ are precisely the ones that are described in Example \ref{oneclass}. Furthermore, the Erd\H{o}s-Ko-Rado sets of type $I_{2n,2n}$ are the ones that are described in Example \ref{secondexample} and the Erd\H{o}s-Ko-Rado sets of type $I_{0,0}$ are the point-pencils.
\end{example}

In the previous example we have introduced several types of Erd\H{o}s-Ko-Rado sets, $I_{k,k}$, whose size is larger than the size of a point-pencil. These are however not the only ones. We shall give two more examples.

\begin{example}\label{iik}
  Again we denote the two classes of generators on $\mathcal{Q}^{+}(4n+1,q)$ by $\Omega_{1}$ and $\Omega_{2}$. Let $\pi$ be a generator of class $\Omega_{1}$ and let $\tau$ be a fixed $k$-space in $\pi$, $0\leq k\leq 2n-2$. Let $\mathcal{S}$ be the union of the set of generators of $\Omega_{1}$ that are not skew to $\tau$ or that meet $\pi$ in a subspace of dimension $i\geq2$, and the set of generators of $\Omega_{2}$ through $\tau$ meeting $\pi$ in a subspace of dimension $2n-1$. It is immediate that the generators in $\mathcal{S}$ pairwise intersect, and hence $\mathcal{S}$ is an Erd\H{o}s-Ko-Rado set. We denote this type of Erd\H{o}s-Ko-Rado sets by $II_{k}$. Its maximality can be proven by arguments similar to the arguments in the proof of the maximality in Example \ref{ikj}.
  \par In the above definition we imposed $k\leq 2n-2$. For $k=2n-1$ this definition gives rise to the Erd\H{o}s-Ko-Rado set described in Example \ref{secondexample}; for $k=2n$ this definition gives rise to the Erd\H{o}s-Ko-Rado set described in Example \ref{oneclass}.
  \par We now count the number of generators in an Erd\H{o}s-Ko-Rado set $\mathcal{S}$ of type $II_{k}$:
  \begin{align*}
    |\mathcal{S}|&=\sum^{n}_{i=1}\gs{2n+1}{2i+1}{q}b^{2n-2i-1}_{2n-2i-1}+\gs{k+1}{1}{q}b^{2n-1}_{2n-1}+\gs{2n-k}{1}{q}\\
    &=\sum^{n}_{i=1}\gs{2n+1}{2i+1}{q}q^{\binom{2(n-i)}{2}}+\gs{k+1}{1}{q}q^{2n^{2}-n}+\gs{2n-k}{1}{q}\;.\\
  \end{align*}
  It can be calculated that an Erd\H{o}s-Ko-Rado set of type $II_{k}$ contains more elements than an Erd\H{o}s-Ko-Rado set $\mathcal{S}$ of type $II_{k'}$ if and only if $k>k'$. Therefore, we calculate the size of an Erd\H{o}s-Ko-Rado set $\mathcal{S}$ of type $II_{2n-2}$:
  \begin{align*}
    |\mathcal{S}|&=\sum^{n}_{i=1}\gs{2n+1}{2i+1}{q}q^{\binom{2(n-i)}{2}}+\gs{2n-1}{1}{q}q^{2n^{2}-n}+\gs{2}{1}{q}\\
    &=\sum^{n-1}_{i=1}q^{\binom{2(n-i)}{2}}\left(q^{2(n-i)}\gs{2n}{2n-2i}{q}+\gs{2n}{2n-2i-1}{q}\right)+1+\frac{q^{2n-1}-1}{q-1}q^{2n^{2}-n}+q+1\\
    &=\sum^{2n-2}_{j=0}\gs{2n}{j}{q}q^{\binom{j}{2}}q^{j}+\frac{q^{2n-1}-1}{q-1}q^{2n^{2}-n}+q+1\\
    &=\sum^{2n}_{j=0}\gs{2n}{j}{q}q^{\binom{j}{2}}q^{j}-q^{2n^{2}+n}-\frac{q^{2n}-1}{q-1}q^{2n^{2}-n}+\frac{q^{2n-1}-1}{q-1}q^{2n^{2}-n}+q+1\\
    &=\prod^{2n}_{i=1}(q^{i}+1)-q^{2n^{2}+n}-q^{2n^{2}+n-1}+q+1\in\Theta(q^{2n^{2}+n-2})\;.
  \end{align*}
  Analogously, the size of an Erd\H{o}s-Ko-Rado set $\mathcal{S}$ of type $II_{0}$ can be calculated. We find
  \[
    |\mathcal{S}|=\prod^{2n}_{i=1}(q^{i}+1)-\frac{(q^{2n}-1)(q^{2n^{2}-n+1}-1)}{q-1}\in\Theta(q^{2n^{2}+n-3})\;.
  \]
\end{example}

\begin{example}
  Before introducing the example, we recall the \emph{triality} map for $\mathcal{Q}^{+}(7,q)$, which has its origins in \cite{tit}; we follow the approach from \cite{lt}. Denote the two generator classes of $\mathcal{Q}^{+}(7,q)$ by $\Omega'_{1}$ and $\Omega'_{2}$. Let $\mathcal{P}$ be the set of points on $\mathcal{Q}^{+}(7,q)$ and let $\mathcal{L}$ be the set of lines on $\mathcal{Q}^{+}(7,q)$. Note that $|\mathcal{P}|=|\Omega'_{1}|=|\Omega'_{2}|$. A $D_{4}$-geometry $\mathcal{G}$ can be constructed as follows. The elements of $\mathcal{P}$ are the $0$-points, the elements of $\Omega_{i}$ are the $i$-points, $i=1,2$, and the elements of $\mathcal{L}$ are the lines. Incidence is defined by symmetrized containment for all combinations of two groups, except for $1$-points and $2$-points; we define a $1$-point and a $2$-point to be incident if they meet in a plane of $\mathcal{Q}^{+}(7,q)$. Every permutation of $\{\mathcal{P},\Omega'_{1},\Omega'_{2}\}$ defines a geometry isomorphic to $\mathcal{G}$. A triality of $\mathcal{G}$ is a map $t$
  \[
    t:\mathcal{L}\rightarrow\mathcal{L},\mathcal{P}\rightarrow\Omega'_{1},\Omega'_{1}\rightarrow\Omega'_{2},\Omega'_{2}\rightarrow\mathcal{P}\;
  \]
  preserving the incidence in $\mathcal{G}$ and such that $t^{3}$ is the identity relation. Such maps are known to exist and are used to construct generalised hexagons.
  \par We use the triality to prove a result about the generators of $\mathcal{Q}^{+}(7,q)$. Let $\pi_{1}$ and $\pi_{2}$ be two disjoint generators of $\Omega'_{2}$ and let $\mathcal{S'}$ be the set of all generators of $\Omega'_{2}$ meeting both $\pi_{1}$ and $\pi_{2}$ in a line. Let $\mathcal{S}''$ be the set of all generators of $\Omega'_{2}$ having a nonempty intersection with all elements of $\mathcal{S}'$. We will show that $\mathcal{S}''=\{\pi_{1},\pi_{2}\}$. It is clear that $\pi^{t}_{1}$ and $\pi^{t}_{2}$ are two points not on a line of $\mathcal{L}$. The set $\mathcal{S}'^{t}$ contains all points of $\mathcal{Q}^{+}(7,q)$ that are collinear with both $\pi^{t}_{1}$ and $\pi^{t}_{2}$. Therefore, $\mathcal{S}'^{t}$ is the set of points on a hyperbolic quadric $\mathcal{Q}^{+}(5,q)$ inside $\mathcal{Q}^{+}(7,q)$. The only points collinear with all points of $\mathcal{S}'^{t}$ are the points $\pi^{t}_{1}$ and $\pi^{t}_{2}$ themselves. Hence, $\mathcal{S}''^{t}=\{\pi^{t}_{1},\pi^{t}_{2}\}$. The statement follows. Note that we can replace $\Omega'_{2}$ by $\Omega'_{1}$ in the statement; replacing $t$ by $t^{2}$, this proof continues.
  \par Now, we consider the hyperbolic quadric $\mathcal{Q}^{+}(4n+1,q)$, $n\geq2$. Denote the two classes of generators by $\Omega_{1}$ and $\Omega_{2}$. Let $\pi$ and $\pi'$ be two generators of $\Omega_{1}$ meeting in a $(2n-4)$-space $\tau$. Let $\mathcal{S}$ be the set containing $\pi$, $\pi'$ and all generators of $\Omega_{2}$ meeting $\pi$ and $\pi'$. It is clear that $\mathcal{S}$ is an Erd\H{o}s-Ko-Rado set. We denote this type of Erd\H{o}s-Ko-Rado sets by $III$.
  \par We prove that an Erd\H{o}s-Ko-Rado set $\mathcal{S}$ of type $III$ is maximal. It is obvious that no generators of $\Omega_{2}$ that are not in $\mathcal{S}$ extend $\mathcal{S}$. Let $\pi''$ be a generator of $\Omega_{1}$ meeting all generators of $\mathcal{S}$. We note that $\mathcal{S}$ contains all generators of $\Omega_{2}$ having a nonempty intersection with $\tau$. We show that $\pi''$ has to contain $\tau$. If $\pi''$ does not contain $\tau$, then we can find a point $P\in\tau\setminus\pi''$. Through $P$ we can find a generator $\sigma'\in\Omega_{2}$ skew to $\pi''$ (by Remark \ref{skewtogenerator} there are $b^{2n-1}_{2n-1}$ such generators). This is a contradiction since $\sigma'$ clearly is contained in $\mathcal{S}$. Hence, we know that $\pi''$ contains $\tau$. Now we consider the tangent space $T_{\tau}$ in $\tau$ to $\mathcal{Q}^{+}(4n+1,q)$. The intersection $T_{\tau}\cap\mathcal{Q}^{+}(4n+1,q)$ is a cone with vertex $\tau$ and base a hyperbolic quadric $\mathcal{Q}^{+}(7,q)$. The generators $\pi$, $\pi'$ and $\pi''$ intersect this base in $\overline{\pi}$, $\overline{\pi}'$ and $\overline{\pi}''$, respectively. These are generators of the hyperbolic quadric $\mathcal{Q}^{+}(7,q)$ of the same class, say $\Omega'_{2}$. Let $\sigma$ be a generator of $\mathcal{Q}^{+}(7,q)$ of class $\Omega'_{2}$, meeting both $\overline{\pi}$ and $\overline{\pi}'$ in a line. We know that there are $b^{2n-4}_{2n-4}$ generators of $\mathcal{Q}^{+}(4n+1,q)$ through $\sigma$ disjoint to $\tau$, necessarily all of class $\Omega_{2}$; these generators are contained in $\mathcal{S}$. Hence $\overline{\pi}''$ has to meet all generators of $\mathcal{Q}^{+}(7,q)$ of class $\Omega'_{2}$ that meet both $\overline{\pi}$ and $\overline{\pi}'$ in a line, as $\sigma$ could be arbitrarily chosen. By the above observation on $\mathcal{Q}^{+}(7,q)$, we know that $\overline{\pi}''\in\{\overline{\pi},\overline{\pi}'\}$. Hence $\pi''\in\{\pi,\pi'\}$, and consequently $\mathcal{S}$ is maximal.
  \par We count the number of generators in an Erd\H{o}s-Ko-Rado set $\mathcal{S}$ of type $III$.
  \begin{align*}
    |\mathcal{S}|&=2+\left(\prod^{2n}_{i=1}(q^{i}+1)-w_{2n,2n-4}\right)+\gs{4}{2}{q}q^{4(2n-3)}b^{2n-4}_{2n-4}\\
    &=2+\prod^{2n}_{i=1}(q^{i}+1)-\left(\prod^{3}_{i=1}(q^{i}+1)\right)q^{(n+2)(2n-3)}+(q^{2}+1)(q^{2}+q+1)q^{(n+2)(2n-3)}\\
    &=\prod^{2n}_{i=1}(q^{i}+1)-q^{2n^{2}+n-6}(q^{6}+q^{5}+q^{3}-q^{2})+2\in\Theta(q^{2n^{2}+n-2})
  \end{align*}
\end{example}

\begin{remark}
  We already noted that the size of the largest maximal Erd\H{o}s-Ko-Rado set is of order $\Theta(q^{2n^{2}+n})$ and the size of the second largest maximal Erd\H{o}s-Ko-Rado set is of order $\Theta(q^{2n^{2}+n-1})$. In the previous Examples we have described three types of maximal Erd\H{o}s-Ko-Rado sets of the next order $\Theta(q^{2n^{2}+n-2})$, namely $I_{2n-2,2n-2}$, $II_{2n-2}$ and $III$. However, it should be noted that the Erd\H{o}s-Ko-Rado sets of type $I_{2n-2,2n-2}$ and Erd\H{o}s-Ko-Rado sets of type $II_{2n-2}$ are the same ones. This is an exceptional case; this pattern does not continue for other Erd\H{o}s-Ko-Rado sets of type $I_{k,k}$ and $II_{k'}$. It can be easily calculated that the Erd\H{o}s-Ko-Rado sets of type $I_{2n-2,2n-2}$ ($II_{2n-2}$) are larger than the Erd\H{o}s-Ko-Rado sets of type $III$.
  \par Calculations in Example \ref{ikj} and Example \ref{iik} show that there are many different Erd\H{o}s-Ko-Rado sets whose size is larger than the size of the point-pencil. So, a complete classification of all Erd\H{o}s-Ko-Rado sets whose size is at least the size of a point-pencil is out of sight for the moment.
\end{remark}

\paragraph*{Acknowledgement:} The research of the author is supported by FWO-Vlaanderen (Research Foundation - Flanders). A part of this research was done while the author was visiting the Eötvös Lor\'and University in Budapest. He wants to thank the members of the Department of Computer Science for their hospitality. He also wants to thank the `Fonds professor Frans Wuytack' for their financial support for this visit.

Address: Maarten De Boeck, UGent, Department of Mathematics, Krijgslaan 281-S22, 9000 Gent, Flanders, Belgium.
\par Email address: mdeboeck@cage.ugent.be

\end{document}